\documentclass[12pt]{amsart}
\usepackage{mathrsfs}

\usepackage{amssymb,amsmath,graphicx,color,textcomp, amsthm,bbm,bbold, enumerate,booktabs}
\usepackage{chngcntr}
\usepackage{apptools}
\AtAppendix{\counterwithin{theorem}{section}}
\definecolor{Red}{cmyk}{0,1,1,0}

\definecolor{verde}{cmyk}{1,0,1,0}

\definecolor{loka}{cmyk}{.5,0,1,.5}
\definecolor{azul}{cmyk}{1,1,0,0}

\numberwithin{equation}{section}

\newcommand{\be}{\begin{equation}}
\newcommand{\ee}{\end{equation}}

\newtheorem{theorem}{Theorem}

\newtheorem{lemma}{Lemma}

\newtheorem{axiom}{Axiom}
\newtheorem{obs}{Observation}

\begin{document}
\title{A new fractional derivative of variable order with non-singular kernel and fractional differential equations}
\author{J. Vanterler da C. Sousa$^1$}
\address{$^1$ Department of Applied Mathematics, Institute of Mathematics,
 Statistics and Scientific Computation, University of Campinas --
UNICAMP, rua S\'ergio Buarque de Holanda 651,
13083--859, Campinas SP, Brazil\newline
e-mail: {\itshape \texttt{ra160908@ime.unicamp.br, capelas@ime.unicamp.br }}}

\author{E. Capelas de Oliveira$^1$}

\begin{abstract} In this paper, we introduce two new non-singular kernel fractional derivatives and present a class of other fractional derivatives derived from the new formulations. We present some important results of uniformly convergent sequences of continuous functions, in particular the Comparison's principle, and others that allow, the study of the limitation of fractional nonlinear differential equations.

\vskip.5cm
\noindent
\emph{Keywords}: Fractional derivative of variable order, Non-singular kernel, Comparison's principle, Fractional differential equation
\newline 
MSC 2010 subject classifications. 26A33; 33BXX; 34A07 .
\end{abstract}
\maketitle


\section{Introduction} 
The topic of fractional calculus is currently one of the most studied, not only by mathematicians, but by physicists, chemists, and various engineers, among others, for their many applications in modeling real phenomena. In fact, when we consider non-integer order derivatives, we have been able, in some studies, to better adapt the theoretical model to the experimental data, thus predicting better the future dynamics of the process \cite{CM,KDI,ELM,HER,RMK,Ze,F5,F6}. A problem that arises in this study is the innumerable definitions of fractional derivative operators, and with this, the choice of the best operator for the case under study. One way to overcome this problem is to consider more general definitions, where the usual ones are particular cases. In this sense, there are currently three branches of fractional derivatives: the fractional derivatives with singular kernel, the so-called "traditional" fractional derivatives, among which we mention: $\psi-$Caputo, $\psi-$Riemann-Liouville, $\psi-$Hilfer, Hadamard, Caputo-Hadamard, Riesz, among others \cite{Ze1,Ze2,RCA,AHMJ}.

On the other hand, recently, two new classes of fractional derivatives have appeared by means of limit, first of all, that said local, we mention some formulations: conformable fractional derivative, $\mathcal{V}-$fractional derivative, $M-$derivative, among others \cite{Ze2,Ze5,RMAM,UNK}.

As we have seen, the number of formulations of fractional derivatives with singular kernel or by limit is considerable, however, it does not make it possible to remedy the diversity of existing problems. In this sense, Caputo and Fabrizio \cite{cap}, proposed a new fractional derivative with non-singular kernel, admitting two different representations, temporal and spatial variables. The interest in this new approach is due to use a model that describes the behavior of classical viscoelastic materials, thermal media, electromagnetic systems, among others. Other fractional derivative formulations with non-singular kernel and applications can be found in references \cite{xiao1,xiao2,atang}.

The motivation for the realization of this paper comes from the $\psi-$Caputo and $\psi-$Riemann-Liouville fractional derivatives and the Riemann-Liouville fractional integral with respect to another function with variable order, in proposing a new fractional derivative with non-singular kernel in a general, way where the usual are particular cases \cite{xiao3}.

The paper is organized as follows: in section 2, we present two new formulations of fractional derivatives with non-singular kernels, by means of the fractional differentiation operators called $\psi-$Riemann-Liouville and $\psi-$Caputo, as well as we study their respective particular cases. We discuss two limit processes $\alpha(t) \rightarrow 1$ and $\alpha(t) \rightarrow 0$ and an observation is made regarding the class of possible fractional operators from the introduced of these new fractional operators. In section 3, we discuss results that guarantee the limitation of new fractional derivatives, as well as results involving uniformly convergent sequences of continuous functions. In this sense, we present and discuss estimates for fractional derivatives. In the section is introduced and proved the result of the fundamental comparison principle for the application of a nonlinear fractional differential equation. Concluding remarks close the paper.

\section{New fractional derivatives}
In this section, by means of the $\psi-$Caputo and $\psi-$Riemann-Liouville fractional derivatives of variable order with respect to another function, we present two new formulations of fractional derivatives of variable order with nonsingular kernel. In this sense, we present some particular cases through remarks and discuss two limiting process es$\alpha(t)\rightarrow 1$ and $\alpha(t)\rightarrow 0$.

We first, consider the one-parameter Mittag-Leffler function, given by \cite{GKAM}
\begin{equation*}
\mathbb{E}_{\beta }\left( z\right) =\overset{\infty }{\underset{k=0}{\sum }}\frac{z^{k}}{\Gamma \left( \beta k+1\right)}, \text{ } \textnormal{Re}(\beta)>0.
\end{equation*}

In particular, taking $\beta \rightarrow 1$, we have $\mathbb{E}_{\beta }\left(z\right) =\exp \left( z\right)$.

Consider the following function, given by
\begin{equation}\label{eq1}
\mathbb{H}_{\gamma ,\beta }^{\alpha \left( t\right) ;\psi }\left( t,\tau \right) :=\mathbb{E}_{\beta }\left[ \frac{-\alpha \left( t\right) \left( \psi \left( t\right) -\psi \left( \tau \right) \right) ^{\gamma }}{1-\alpha \left( t\right) }\right] 
\end{equation}
with $0<\alpha \left( t\right) <1$, $0<\beta,\gamma<1$ $\psi \left( \cdot \right) $ is a positive function and increasing monotone and $\mathbb{E}_{\beta }\left(\cdot\right) $ is a Mittag-Leffler function, which is considered uniformly convergent on the interval $[a,b]=I.$

The $\psi -$Riemann-Liouville fractional integral of variable order, $0<\alpha \left( t\right) <1$, is given by \cite{xiao3}
\begin{equation*}
I_{a+}^{\alpha \left( t\right) ;\psi }f\left( t\right) :=\frac{1}{\Gamma \left( \alpha \left( t\right) \right)}\int_{a}^{t}\psi ^{\prime }\left( \tau \right) \left( \psi \left( t\right) -\psi \left( \tau \right) \right) ^{\alpha \left( \tau \right) -1}f\left( \tau \right) d\tau .
\end{equation*}

The corresponding $\psi-$Riemann-Liouville fractional derivative of variable order $0<\alpha \left( t\right) <1$, is given by \cite{xiao3}
\begin{equation} \label{eq2}
^{RL}\mathcal{D}_{a+}^{\alpha \left( t\right) ;\psi }f\left( t\right) :=\frac{1}{\Gamma \left( 1-\alpha \left( t\right) \right) }\left( \frac{1}{\psi ^{\prime }\left( t\right) }\frac{d}{dt}\right) \int_{a}^{t}\psi ^{\prime }\left( \tau \right) \left( \psi \left( t\right) -\psi \left( \tau \right) \right) ^{-\alpha \left( t\right) }f\left( \tau \right) d\tau 
\end{equation}
with $0<\alpha \left( t\right) <1,$ $a\in \lbrack -\infty, t]$, $f\in H_{1}\left( a,b\right) $ and $b>a.$

On the other hand, the $\psi -$Caputo fractional derivative of variable order $0<\alpha \left( t\right) <1$ is \cite{xiao3}
\begin{equation}\label{eq3}
^{C}\mathcal{D}_{a+}^{\alpha \left( t\right) ;\psi }f\left( t\right) :=\frac{1}{\Gamma \left( 1-\alpha \left( t\right) \right) }\int_{a}^{t}\left( \psi \left( t\right) -\psi \left( \tau \right) \right) ^{-\alpha \left( t\right) }f^{\prime }\left( \tau \right) d\tau,
\end{equation}
with $0<\alpha \left( t\right) <1,$ $a\in \lbrack -\infty ,t],$ $f\in H_{1}\left( a,b\right) $ and $b>a.$

Second, by changing the kernel $\left( \psi \left( t\right) -\psi \left( \tau \right) \right) ^{-\alpha \left( t\right) }$ with the function $\mathbb{H}_{\gamma ,\beta }^{\alpha \left( t\right) ;\psi }\left( t,\tau \right) $ and $\dfrac{1}{\Gamma \left( 1-\alpha \left( t\right) \right) }$ with $\dfrac{M\left( \alpha \left( t\right) \right) }{1-\alpha \left( t\right) }$ in Eq.(\ref{eq2}) and Eq.(\ref{eq3}), we obtain the following two new definitions of time fractional derivative given by 
\begin{equation}\label{eq4}
^{RL}\mathfrak{D}_{a+}^{\alpha \left( t\right) ;\psi }f\left( t\right) :=\frac{M\left( \alpha \left( t\right) \right) }{1-\alpha \left( t\right) }\left( \frac{1}{\psi ^{\prime }\left( t\right) }\frac{d}{dt}\right) \int_{a}^{t}\psi ^{\prime }\left( \tau \right) \mathbb{H}_{\gamma ,\beta }^{\alpha \left( t\right); \psi }\left( t,\tau \right) f\left( \tau \right) d\tau 
\end{equation}
and
\begin{equation}\label{eq5} 
^{C}\mathfrak{D}_{a+}^{\alpha \left( t\right) ;\psi }f\left( t\right) :=\frac{M\left( \alpha \left( t\right) \right) }{1-\alpha \left( t\right) }\int_{a}^{t}\mathbb{H}_{\gamma ,\beta }^{\alpha \left( t\right) ;\psi }\left( t,\tau \right) f^{\prime }\left( \tau \right) d\tau 	
\end{equation}
where $M\left( \alpha \left( t\right) \right) $ is a normalization function such that $M\left( 0\right) =M\left( 1\right) =1,$ and $\psi \left( \cdot \right) $ is a positive function and increasing monotone, such that $\psi ^{\prime }\left( t\right) \neq 0.$ 

Next, we introduce an axiom that makes it possible to recover two limiting cases, specifically the integer-order cases.

\begin{axiom}\label{axioma} Let $0<\alpha \left( t\right) <1,$ $\gamma ,\beta \in \mathbb{R}$ and $\psi\in C^{1}[a,b] $ an increasing function such that $\psi ^{\prime }\left( t\right) \neq 0$ on $\left( a,b\right) $. Let $\mathbb{H}_{\gamma ,\beta }^{\alpha \left( t\right) ;\psi }\left(t,\tau \right) $ be a function as given by \textnormal{ Eq.(\ref{eq1})}, then
\begin{equation}\label{eq6}
\underset{\alpha \left( t\right) \rightarrow 1}{\lim }\frac{M\left( \alpha \left( t\right) \right) }{1-\alpha \left( t\right) }\left( \frac{1}{\psi ^{\prime }\left( t\right) }\frac{d}{dt}\right) \int_{a}^{t}\psi ^{\prime }\left( \tau \right) \mathbb{H}_{\gamma ,\beta }^{\alpha \left( t\right) ;\psi }\left( t,\tau \right) f\left( \tau \right) d\tau =f^{\prime }\left(
t\right) 
\end{equation}
and
\begin{equation}\label{eq7}
\underset{\alpha \left( t\right) \rightarrow 1}{\lim }\frac{M\left( \alpha \left( t\right) \right) }{1-\alpha \left( t\right) } \int_{a}^{t} \mathbb{H}_{\gamma ,\beta }^{\alpha \left( t\right) ;\psi }\left( t,\tau \right) f\left( \tau \right) d\tau =f^{\prime }\left(
t\right) .
\end{equation}
\end{axiom}

Now, observe that taking limit $\alpha \left( t\right) \rightarrow 1,$ on both sides of Eq.(\ref{eq4}), Eq.(\ref{eq5}) and by Axiom \ref{axioma}, i.e., Eq.(\ref{eq6}) and Eq.(\ref{eq7}), we have 
\begin{equation*}
\underset{\alpha \left( t\right) \rightarrow 1}{\lim }^{RL}\mathfrak{D}_{a+}^{\alpha \left( t\right) ;\psi }f\left( t\right) =f^{\prime }\left( t\right) 
\end{equation*}
and
\begin{equation*}
\underset{\alpha \left( t\right) \rightarrow 1}{\lim }^{C}\mathfrak{D}_{a+}^{\alpha \left( t\right) ;\psi }f\left( t\right) =f^{\prime }\left( t\right) .
\end{equation*}

Note that, in the limit $\alpha(t)\rightarrow 0$, we have
\begin{eqnarray}\label{eq8}
&&\underset{\alpha \left( t\right) \rightarrow 0}{\lim }\mathbb{H}_{\gamma ,\beta }^{\alpha \left( t\right) ;\psi }\left( t,\tau \right) \notag \\ &=&\underset{\alpha \left( t\right) \rightarrow 0}{\lim }\overset{\infty }{\underset{k=0}{\sum }}\dfrac{\left[ \alpha \left( t\right) \left( \psi \left( t\right) -\psi \left( a\right) \right) ^{\gamma }\right] ^{k}}{\left( 1-\alpha \left( t\right) \right) ^{k}\Gamma \left( \beta k+1\right) }  \notag \\ &=&\underset{\alpha \left( t\right) \rightarrow 0}{\lim }\left\{ 
\begin{array}{c}
1+\dfrac{\left[ \alpha \left( t\right) \left( \psi \left( t\right) -\psi \left( a\right) \right) ^{\gamma }\right] }{\left( 1-\alpha \left( t\right) \right)\Gamma \left( \beta +1\right) }+\dfrac{\left[ \alpha \left( t\right) \left( \psi \left( t\right) -\psi \left( a\right) \right) ^{\gamma }\right] ^{2}}{\left( 1-\alpha \left( t\right) \right) ^{2}\Gamma\left( 2\beta+1\right) } \\ 
+\dfrac{\left[ \alpha \left( t\right) \left( \psi \left( t\right) -\psi \left( a\right) \right) ^{\gamma }\right] ^{3}}{\left( 1-\alpha \left( t\right) \right) ^{3}\Gamma \left( 3\beta +1\right) }+\cdot \cdot \cdot +\dfrac{\left[ \alpha \left( t\right) \left( \psi \left( t\right) -\psi \left( a\right) \right) ^{\gamma }\right] ^{n}}{\left( 1-\alpha \left( t\right)
\right) ^{n}\Gamma \left( n\beta+1\right) }+\cdot \cdot \cdot  \notag
\end{array}
\right\} =1.\\
\end{eqnarray}

Taking the limit $\alpha \left( t\right) \rightarrow 0,$ on both sides of Eq.(\ref{eq4}) and by Eq.(\ref{eq8}), we obtain
\begin{eqnarray*}
&&\underset{\alpha \left( t\right) \rightarrow 0}{\lim }\text{ }^{RL}\mathfrak{D}_{a+}^{\alpha \left( t\right) ;\psi }f\left( t\right) \notag \\
&=&\underset{\alpha \left( t\right) \rightarrow 0}{\lim }\left\{ \frac{M\left( \alpha \left( t\right) \right) }{1-\alpha \left( t\right) }\left( \frac{1}{\psi ^{\prime }\left( t\right) }\frac{d}{dt}\right) \right\} \int_{a}^{t}\psi
^{\prime }\left( \tau \right) \underset{\alpha \left( t\right) \rightarrow 0} {\lim }\mathbb{H}_{\gamma ,\beta }^{\alpha \left( t\right) ;\psi }\left( t,\tau \right) f\left( \tau \right) d\tau  \\
&=&\left( \frac{1}{\psi ^{\prime }\left( t\right) }\frac{d}{dt}\right) \int_{a}^{t}\psi ^{\prime }\left( \tau \right) f\left( \tau \right) d\tau  \\ &=&f\left( t\right) .
\end{eqnarray*}

On the other hand, taking the limit $\alpha \left( t\right) \rightarrow 0,$ on both sides of Eq.(\ref{eq5}) and by Eq.(\ref{eq8}), we get
\begin{eqnarray*}
\underset{\alpha \left( t\right) \rightarrow 0}{\lim }\text{ } ^{C}\mathfrak{D}_{a+}^{\alpha \left( t\right) ;\psi }f\left( t\right)  &=&\underset{\alpha \left( t\right) \rightarrow 0}{\lim }\left\{ \frac{M\left( \alpha \left( t\right) \right) }{1-\alpha \left( t\right) }\right\} \int_{a}^{t}\underset{\alpha \left( t\right) \rightarrow 0}{\lim }\mathbb{H}_{\gamma ,\beta
}^{\alpha \left( t\right) ;\psi }\left( t,\tau \right) f^{\prime }\left( \tau \right) d\tau  \\
&=&\int_{a}^{t}f^{\prime }\left( \tau \right) d\tau  \\ &=&f\left( t\right) -f\left( a\right) .
\end{eqnarray*}

By means of the operators $^{RL}\mathfrak{D}_{a+}^{\alpha \left( t\right) ;\psi }\left( \cdot \right) $ and $^{C}\mathfrak{D}_{a+}^{\alpha \left( t\right) ;\psi }\left( \cdot \right) $, we recover some particular cases, as follow.

\begin{obs}\label{ob1} First, for the fractional derivative $^{RL}\mathfrak{D}_{a+}^{\alpha \left( t\right) ;\psi }\left( \cdot \right) ,$ we have the following particular cases:
\begin{enumerate}
\item Taking $\beta =\gamma =\alpha $ and $\psi \left( t\right) =t$ in \textnormal{Eq.(\ref{eq4})}, we have the fractional derivative, given by \textnormal{\cite{xiao2}}
\begin{equation*}
^{RL}\mathfrak{D}_{a+}^{\alpha \left( t\right) ;t}f\left( t\right)=\frac{M\left( \alpha(t) \right) }{1-\alpha(t) }\left( \frac{d}{dt}\right) \int_{a}^{t}\mathbb{E}_{\alpha }\left( -\frac{\alpha(t) \left( t-\tau \right) ^{\alpha }}{1-\alpha(t) }\right) f\left( \tau \right) d\tau \equiv \mathscr{D}_{a+}^{\alpha }f\left( t\right) .
\end{equation*}

\item Taking $\alpha \left( t\right) =\beta =\gamma =\alpha $ and $\psi \left( t\right) =t$ in \textnormal{Eq.(\ref{eq4})}, we have the so-called Atangana fractional derivative, given by \textnormal{\cite{atang}}
\begin{equation*}
^{RL}\mathfrak{D}_{a+}^{\alpha \left( t\right) ;t}f\left( t\right)=\frac{M\left( \alpha \right) }{1-\alpha }\left( \frac{d}{dt}\right) \int_{a}^{t}\mathbb{E}_{\alpha }\left( -\frac{\alpha \left( t-\tau \right) ^{\alpha }}{1-\alpha }\right)
f\left( \tau \right) d\tau \equiv\text{ }^{AT}\mathfrak{D}_{a+}^{\alpha }f\left( t\right) .
\end{equation*}

\item Taking $\alpha \left( t\right) =\alpha $ and $\beta $ $=\gamma =1$ and $\psi \left( t\right) =t$ in \textnormal{Eq.(\ref{eq4})}, we have the so-called Yang-Machado fractional derivative, given by \textnormal{\cite{xiao1}}
\begin{equation*}
^{RL}\mathfrak{D}_{a+}^{\alpha \left( t\right) ;t}f\left( t\right)=\frac{M\left( \alpha \right) }{1-\alpha }\left( \frac{d}{dt}\right) \int_{a}^{t}\exp \left( -\frac{\alpha \left( t-\tau \right) }{1-\alpha }\right) f\left( \tau \right) d\tau \equiv\text{ }^{YM}\mathcal{D}_{a+}^{\alpha }f\left( t\right) .
\end{equation*}

\item Taking $\alpha \left( t\right) =\alpha $, $\beta $ $=\gamma =1,$ $M\left( \alpha \right) =1$ and $\psi \left( t\right) =t$ in \textnormal{Eq.(\ref{eq4})}, we have the fractional derivative, given by \textnormal{\cite{xiao1}}
\begin{equation*}
^{RL}\mathfrak{D}_{a+}^{\alpha \left( t\right) ;t}f\left( t\right)=\frac{1}{1-\alpha }\left( \frac{d}{dt}\right) \int_{a}^{t}\exp \left( -\frac{\alpha \left( t-\tau \right) }{1-\alpha }\right) f\left( \tau \right) d\tau \equiv\mathcal{D}_{a+}^{\alpha }f\left( t\right) .
\end{equation*}

\item Note that, as $\psi \left( t\right) $ is an arbitrary function, if we take $\psi \left( t\right) =\ln t$, we obtain
\begin{equation*}
^{RL}\mathfrak{D}_{a+}^{\alpha \left( t\right) ;\ln t}f\left( t\right)=\frac{M\left( \alpha \left( t\right) \right) }{1-\alpha \left( t\right) }\left( t\frac{d}{dt}\right) \int_{a}^{t}\mathbb{E}_{\beta }\left( -\frac{\alpha \left( t\right) \ln \left( \frac{t}{\tau }\right) ^{\gamma }}{1-\alpha \left( t\right) }\right) f\left( \tau \right) \frac{d\tau}{\tau}.
\end{equation*}
\end{enumerate}
\end{obs}

\begin{obs}\label{ob2} On the other hand, for the fractional derivative $^{C}\mathfrak{D}_{a+}^{\alpha \left( t\right) ;\psi }\left( \cdot \right) ,$ we have the following particular cases:

\begin{enumerate}
\item  For $\beta =\gamma =\alpha $ and $\psi \left( t\right) =t$ in \textnormal{Eq.(\ref{eq5})}, we have the fractional derivative, given by \textnormal{\cite{xiao2}}
\begin{equation*}
^{C}\mathfrak{D}_{a+}^{\alpha \left( t\right) ;t}f\left( t\right) =\frac{M\left( \alpha(t) \right) }{1-\alpha(t)}\int_{a}^{t}\mathbb{E}_{\alpha }\left( -\frac{\alpha(t)\left( t-\tau \right) ^{\alpha }}{1-\alpha }\right) f^{\prime }\left( \tau \right) d\tau \equiv\mathcal{D}_{a+}^{\alpha(t) }f\left( t\right) .
\end{equation*}

\item  For $\alpha \left( t\right) =\beta =\gamma =\alpha $ and $\psi \left( t\right) =t$ in \textnormal{Eq.(\ref{eq5})}, we have the so-called Atangana fractional derivative, given by \textnormal{\cite{atang}}
\begin{equation*}
^{C}\mathfrak{D}_{a+}^{\alpha \left( t\right) ;t}f\left( t\right) =\frac{M\left( \alpha \right) }{1-\alpha }\int_{a}^{t}\mathbb{E}_{\alpha }\left( -\frac{\alpha \left( t-\tau \right) ^{\alpha }}{1-\alpha }\right) f^{\prime }\left( \tau \right) d\tau \equiv\text{ }^{AT}\mathcal{D}_{a+}^{\alpha }f\left( t\right) .
\end{equation*}

\item For $\alpha \left( t\right) =\alpha $ and $\beta $ $=\gamma =1$ and $ \psi \left( t\right) =t$ in \textnormal{Eq.(\ref{eq5})}, we have the Caputo-Fabrizio fractional derivative, given by \textnormal{\cite{cap}}
\begin{equation*}
^{C}\mathfrak{D}_{a+}^{\alpha \left( t\right) ;t}f\left( t\right) =\frac{M\left( \alpha \right) }{1-\alpha }\int_{a}^{t}\exp \left( -\frac{\alpha \left( t-\tau \right) }{1-\alpha }\right) f^{\prime }\left( \tau \right) d\tau \equiv\text{ }^{CF}\mathcal{D}_{a+}^{\alpha }f\left( t\right) .
\end{equation*}

\item For $\alpha \left( t\right) =\alpha $, $\beta $ $=\gamma =1,$ $M\left( \alpha \right) =1$ and $\psi \left( t\right) =t$ in \textnormal{Eq.(\ref{eq5})}, we have the fractional derivative, given by \textnormal{\cite{xiao1}}
\begin{equation*}
^{C}\mathfrak{D}_{a+}^{\alpha \left( t\right) ;t}f\left( t\right) :=\frac{1}{1-\alpha }\int_{a}^{t}\exp \left( -\frac{\alpha \left( t-\tau \right) }{1-\alpha }\right) f^{\prime }\left( \tau \right) d\tau \equiv\mathcal{D}_{a+}^{\alpha }f\left( t\right) .
\end{equation*}

\item Here, we can also choose, other function $\psi \left( t\right)$, in particular $\psi \left( t\right) =\sin t$ and obtain another possible fractional derivative over time.
\end{enumerate}
\end{obs}

\begin{obs} In these two above \textnormal{Remark \ref{ob1}} and \textnormal{ Remark \ref{ob2}}, we recover particular cases of some fractional derivatives with non-singular kernels. Also, it is possible, from the choice of the $\psi(\cdot)$ function, to obtain formulation another fractional derivative. Note that,	 the definitions of the fractional derivatives \textnormal{ Eq.(\ref{eq4})} and \textnormal{ Eq.(\ref{eq5})}, were obtained  by fractional derivatives $\psi-$Caputo and $\psi-$Riemann$-$Liouville and the Riemann$-$Liouville fractional integral with respect to another function \textnormal{\cite{RCA,AHMJ}}. From these fractional derivatives, we obtain other formulations, which, as a consequence, could introduce another fractional derivative formulation with non-singular kernel, performing the same procedure used for the fractional derivatives\textnormal{ Eq.(\ref{eq4})} and \textnormal{ Eq.(\ref{eq5})}. So, in fact, the formulations presented here are quite general.
\end{obs}

\section{Miscellaneous} 
In this section, we present some results through the fractional derivatives as introduced in the previous section. As a first result, we discuss the bounded of fractional derivatives, as well as, results that involve uniformly convergent sequences of continuous functions, both for the fractional derivatives, and for two variations of integrals important to obtain some results. In this sense, we present and discuss estimates of fractional derivatives.

\begin{theorem} Let $0<\alpha \left( t\right) <1$ and $f,\psi \in C\left[a, b\right] $ be two functions such that $\psi $ is increasing and $\psi ^{\prime }\left( t\right) \neq 0,$ for all $t\in \lbrack a,b]$ a closed interval. Then the inequality is obtained on $\left[ a,b\right]$,
\begin{equation}\label{eq9}
\left\Vert ^{RL}\mathfrak{D}_{a+}^{\alpha \left( t\right) ;\psi }f\right\Vert \leq \frac{M\left( \alpha \left( b\right) \right) }{1-\alpha \left( b\right) }\left\Vert f\right\Vert 
\end{equation} 
and
\begin{equation}\label{eq10}
\left\Vert ^{C}\mathfrak{D}_{a+}^{\alpha \left( t\right) ;\psi }f\right\Vert \leq \frac{M\left( \alpha \left( b\right) \right) }{1-\alpha \left( b\right) }\left\Vert f\right\Vert 
\end{equation}
where $M\left( \alpha \left( b\right) \right) $ is a normalization function such that $M\left( 0\right) =M\left( 1\right) =1.$
\end{theorem}
\begin{proof} We present a proof of Eq.(\ref{eq9}) and to prove Eq.(\ref{eq10}), follows in a similar way. Then, by definition Eq.(\ref{eq4}) and as $\mathbb{H}_{\gamma ,\beta }^{\alpha \left( t\right) ;\psi }\left( t,\tau \right) \leq 1,$ for $0<\alpha \left( t\right) <1,$ we can write
\begin{eqnarray*}
\left\Vert ^{RL}\mathfrak{D}_{a+}^{\alpha \left( t\right) ;\psi }f\right\Vert  &=& \underset{t\in \left[ a,b\right] }{\max }\frac{M\left( \alpha \left( t\right) \right) }{1-\alpha \left( t\right) }\left\vert \left( \frac{1}{\psi
^{\prime }\left( t\right) }\frac{d}{dt}\right) \int_{a}^{t}\psi ^{\prime }\left( \tau \right) \mathbb{H}_{\gamma ,\beta }^{\alpha \left( t\right) ;\psi }\left( t,\tau \right) f\left( \tau \right) d\tau \right\vert  \\ 
&\leq &\underset{t\in \left[ a,b\right] }{\max }\frac{M\left( \alpha \left( t\right) \right) }{1-\alpha \left( t\right) }\left\vert \left( \frac{1}{\psi ^{\prime }\left( t\right) }\frac{d}{dt}\right) \int_{a}^{t}\psi ^{\prime 
}\left( \tau \right) f\left( \tau \right) d\tau \right\vert  \\ &\leq &\frac{M\left( \alpha \left( b\right) \right) }{1-\alpha \left( b\right) }\underset{t\in \left[ a,b\right] }{\max }\left\vert f\left( t\right) \right\vert  \\
&=&\frac{M\left( \alpha \left( b\right) \right) }{1-\alpha \left( b\right) }\left\Vert f\right\Vert.
\end{eqnarray*}
\end{proof}

\begin{theorem} Let $0<\alpha \left( t\right) <1$ and $f,\psi \in C\left[ a,b\right] $ be two functions such that $\psi $ is increasing and $\psi ^{\prime }\left( t\right) \neq 0$, for all $t\in \left[ a,b\right] $ a closed interval. The $\psi-$Riemann-Liouville and $\psi-$Caputo fractional derivatives, that is to say for a given couple function $f$ and $h$, the following inequalities can be established:
\begin{equation}\label{eq11}
\left\Vert ^{RL}\mathfrak{D}_{a+}^{\alpha \left( t\right) ;\psi }f-\text{ } ^{RL}\mathfrak{D}_{a+}^{\alpha \left( t\right) ;\psi }g\right\Vert \leq N\left\Vert f-g\right\Vert 
\end{equation}
and
\begin{equation}\label{ea12}
\left\Vert ^{C}\mathfrak{D}_{a+}^{\alpha \left( t\right) ;\psi }f-\text{ } ^{C}\mathfrak{D}_{a+}^{\alpha \left( t\right) ;\psi }g\right\Vert \leq N\left\Vert f-g\right\Vert ,
\end{equation}
with $N=\theta _{1}\dfrac{M\left( \alpha \left( b\right) \right) }{1-\alpha \left( b\right) }\text{ }\mathbb{H}_{\gamma ,\beta }^{\alpha \left( b\right) ;\psi }\left( b,a\right) \left( b-a\right)$ where $\theta_{1}>0$.
\end{theorem}

\begin{proof} In fact, by definition of fractional operator $^{RL}\mathfrak{D}_{a+}^{\alpha \left( t\right) ;\psi }\left( \cdot \right) $ Eq.(\ref{eq4}), we have
\begin{eqnarray*}
&&\left\Vert ^{RL}\mathfrak{D}_{a+}^{\alpha \left( t\right) ;\psi }f-\text{ } ^{RL}\mathfrak{D}_{a+}^{\alpha \left( t\right) ;\psi }g\right\Vert  \\
&=&\underset{t\in \left[ a,b\right] }{\max }\frac{M\left( \alpha \left( t\right) \right) }{1-\alpha \left( t\right) }\left\vert \left( \frac{1}{\psi ^{\prime }\left( t\right) }\frac{d}{dt}\right) \int_{a}^{t}\psi ^{\prime
}\left( \tau \right) \mathbb{H}_{\gamma ,\beta }^{\alpha \left( t\right) ;\psi }\left( t,\tau \right) \left( f\left( \tau \right) -g\left( \tau \right) \right) d\tau \right\vert .
\end{eqnarray*}

Using the Lipschitz condition of the first order derivative, we can find a positive constant such that,
\begin{eqnarray*}
&&\left\Vert ^{RL}\mathfrak{D}_{a+}^{\alpha \left( t\right) ;\psi }f-\text{ } ^{RL}\mathfrak{D}_{a+}^{\alpha \left( t\right) ;\psi }g\right\Vert  \\
&\leq &\theta _{1}\underset{t\in \left[ a,b\right] }{\max }\frac{M\left( \alpha \left( t\right) \right) }{1-\alpha \left( t\right) }\mathbb{H}_{\gamma ,\beta }^{\alpha \left( t\right) ;\psi }\left( t,a\right) \left\vert \frac{1}{\psi
^{\prime }\left( t\right) }\int_{a}^{t}\psi ^{\prime }\left( \tau \right) \left( f\left( \tau \right) -g\left( \tau \right) \right) d\tau \right\vert  \\
&\leq &\theta _{1}\underset{t\in \left[ a,b\right] }{\max }\frac{M\left( \alpha \left( t\right) \right) }{1-\alpha \left( t\right) }\mathbb{H}_{\gamma ,\beta }^{\alpha \left( t\right) ;\psi }\left( t,a\right) \left\vert f\left( t\right) -g\left( t\right) \right\vert \left( t-a\right)  \\ &\leq &\theta _{1}\frac{M\left( \alpha \left( b\right) \right) }{1-\alpha \left( b\right) }\mathbb{H}_{\gamma ,\beta }^{\alpha \left( b\right) ;\psi }\left( b,a\right) \left( b-a\right) \left\Vert f-g\right\Vert  \\ &=&N\left\Vert f-g\right\Vert ,
\end{eqnarray*}
with $N=\theta _{1}\dfrac{M\left( \alpha \left( b\right) \right) }{1-\alpha \left( b\right) }\text{ }\mathbb{H}_{\gamma ,\beta }^{\alpha \left( b\right) ;\psi }\left( b,a\right) \left( b-a\right)$, where $\theta_{1}>0$.

Therefore, we conclude the prove of the Eq.(\ref{eq11}). The prove for the Eq.(\ref{ea12}), can be realize in an analogous way.
\end{proof}

Recently Sousa and Oliveira \cite{Ze1}, introduced the fractional differential operator, the so-called $\psi-$Hilfer fractional differential operator and through it some important results of the fractional calculus was recovered. In particular, one of the results addressed in this paper, involves uniformly convergent sequence of continuous functions on $\left[ a,b\right] ,$ for the following operators $I_{a+}^{\alpha ;\psi }\left( \cdot \right) $ and $^{H}\mathbb{D}_{a+}^{\alpha ,\beta ;\psi }\left( \cdot \right) .$ Next, as motivation, we present some results in this context.

Consider the following integrals 
\begin{equation*}
_{1}I_{\gamma ,\beta }^{\alpha \left( t\right) ;\psi }f\left( t\right) =\int_{a}^{t}\psi ^{\prime }\left( \tau \right) \mathbb{H}_{\gamma ,\beta }^{\alpha \left( t\right) ;\psi }\left( t,\beta \right) f\left( \tau \right) dr
\end{equation*}
and
\begin{equation*}
_{2}I_{\gamma ,\beta }^{\alpha \left( t\right) ;\psi }f\left( t\right) =\int_{a}^{t}\mathbb{H}_{\gamma ,\beta }^{\alpha \left( t\right) ;\psi }\left( t,\beta \right) f^{\prime }\left( \tau \right) dr.
\end{equation*}

Then, we have the corresponding derivatives
\begin{equation}\label{eq13}
^{RL}\mathfrak{D}_{a+}^{\alpha \left( t\right) ;\psi }f\left( t\right) =\frac{M\left(\alpha \left( t\right) \right) }{1-\alpha \left( t\right) }\left( \frac{1}{\psi ^{\prime }\left( t\right) }\frac{d}{dt}\right) \text{ }_{1}I_{\gamma
,\beta }^{\alpha \left( t\right) ;\psi }f\left( t\right) 
\end{equation}
and
\begin{equation}\label{eq14}
^{C}\mathfrak{D}_{a+}^{\alpha \left( t\right) ;\psi }f\left( t\right) =\frac{M\left( \alpha \left( t\right) \right) }{1-\alpha \left( t\right) }\text{ }_{2}I_{\gamma ,\beta }^{\alpha \left( t\right) ;\psi }f\left( t\right) .
\end{equation}

\begin{theorem}\label{teo3} Let $0<\alpha \left( t\right) <1,$ $I=\left[ a,b\right] $ be a finite or infinite interval and $\psi \in C\left[ a,b\right] $ an increasing function such that $\psi ^{\prime }\left( t\right) \neq 0$, for all $t\in I.$ Assume that $\left( f_{k}\right) _{k=1}^{\infty }$ is a uniformly convergent sequence of continuous functions on $\left[ a,b\right] .$ Then, we may interchange the fractional integral operator and the limit process, i.e., 
\begin{equation*}
_{1}I_{\gamma ,\beta }^{\alpha \left( t\right) ;\psi }\underset{k\rightarrow \infty }{\lim }f_{k}\left( t\right) =\underset{k\rightarrow \infty }{\lim }\text{ }_{1}I_{\gamma ,\beta }^{\alpha \left( t\right) ;\psi }f_{k}\left(
t\right) .
\end{equation*}

In particular, the sequence of function $\left( _{1}I_{\gamma ,\beta }^{\alpha \left( t\right) ;\psi }f_{k}\right) _{k=1}^{\infty }$ is uniformly convergent.
\end{theorem}

\begin{proof} In fact, we denote the limit of the sequence $\left( f_{k}\right)$ by $f$. It is well known that if $f$ is continuous we then find
\begin{eqnarray*}
\left\vert _{1}I_{\gamma ,\beta }^{\alpha \left( t\right) ;\psi }f_{k}\left( t\right) -\text{ }_{1}I_{\gamma ,\beta }^{\alpha \left( t\right) ;\psi }f\left( t\right) \right\vert  &\leq &\int_{a}^{t}\psi ^{\prime }\left( \tau
\right) \mathbb{H}_{\gamma ,\beta }^{\alpha \left( t\right) ;\psi }\left( t,\tau \right) \left\vert f_{k}\left( \tau \right) -f\left( \tau \right) \right\vert d\tau  \\
&\leq &\left\Vert f_{k}-f\right\Vert _{\infty }\int_{a}^{t}\psi ^{\prime }\left( \tau \right) d\tau  \\
&\leq &\left( \psi \left( b\right) -\psi \left( a\right) \right) \left\Vert f_{k}-f\right\Vert _{\infty }.
\end{eqnarray*}

As $\ f_{k}$ is a uniformly convergent sequence, we conclude the proof.
\end{proof}

\begin{theorem}\label{teo4} Let $0<\alpha \left( t\right) <1,$ $I=\left[ a,b\right] $ be a finite or infinite interval and $\psi \in C\left[ a,b\right] $ an increasing function such that $\psi ^{\prime }\left( t\right) \neq 0$, for all $t\in I.$ Assume that $\left( f_{k}^{\prime }\right) _{k=1}^{\infty }$ is a uniformly convergent sequence of continuous functions on $\left[ a,b\right]$. Then, we may interchange the fractional integral operator and the limit process, i.e., 
\begin{equation*}
_{2}I_{\gamma ,\beta }^{\alpha \left( t\right) ;\psi }\underset{k\rightarrow \infty }{\lim }f_{k}\left( t\right) =\underset{k\rightarrow \infty }{\lim }\text{ }_{2}I_{\gamma ,\beta }^{\alpha \left( t\right) ;\psi }f_{k}\left(
t\right).
\end{equation*}

In particular, the sequence of function $\left( _{2}I_{\gamma ,\beta }^{\alpha \left( t\right) ;\psi }f_{k}\right) _{k=1}^{\infty }$ is uniformly convergent.
\end{theorem}

\begin{proof} The prove, follows the same steps as in Theorem \ref{teo3}
\end{proof}

\begin{theorem}\label{teo5} Let $0<\alpha \left( t\right) <1,$ $I=\left[ a,b\right] $ be a finite or infinite interval and $\psi \in C\left[ a,b\right] $ an increasing function such that $\psi ^{\prime }\left( t\right) \neq 0$, for all $t\in I.$ Assume that $\left( f_{k}\right) _{k=1}^{\infty }$ is a uniformly convergent sequence of continuous functions on $\left[ a,b\right] $ and $ ^{RL}\mathfrak{D}_{a+}^{\alpha \left( t\right) ;\psi }f_{k}\left( \cdot \right) $ exist for every $k$. Moreover assume that $\left( ^{RL}\mathfrak{D}_{a+}^{\alpha \left( t\right) ;\psi }f_{k}\right) _{k=1}^{\infty }$ is convergent uniformly on $\left[ a+\varepsilon ,b\right] $ for every $\varepsilon >0.$ Then, for every $t\in 
\left[ a,b\right] $ we have
\begin{equation*}
\underset{k\rightarrow \infty }{\lim }\text{ }^{RL}\mathfrak{D}_{a+}^{\alpha \left( t\right) ;\psi }f_{k}\left( t\right) =\text{ }^{RL}\mathfrak{D}_{a+}^{\alpha \left( t\right) ;\psi }\underset{k\rightarrow \infty }{\lim }f_{k}\left( t\right).
\end{equation*}
\end{theorem}

\begin{proof} Using the definition of fractional derivative $^{RL}\mathfrak{D}_{a+}^{\alpha \left( t\right) ;\psi }\left( \cdot \right) $ Eq.(\ref{eq4}) and by Theorem \ref{teo3}, follows
\begin{equation*}
_{1}I_{\gamma ,\beta }^{\alpha \left( t\right) ;\psi }\underset{k\rightarrow \infty }{\lim }f_{k}\left( t\right) =\underset{k\rightarrow \infty }{\lim } \text{ }_{1}I_{\gamma ,\beta }^{\alpha \left( t\right) ;\psi }f_{k}\left(
t\right).
\end{equation*}

On the other hand, by hypotheses $^{RL}\mathfrak{D}_{a+}^{\alpha \left( t\right) ;\psi }\left( \cdot \right) $ is convergent uniformly on $\left[ a+\varepsilon ,b \right] $, for every $\varepsilon >0,$ then we obtain
\begin{eqnarray*}
\underset{k\rightarrow \infty }{\lim }^{RL}\mathfrak{D}_{a+}^{\alpha \left( t\right) ;\psi }f_{k}\left( t\right)  &=&\underset{k\rightarrow \infty }{\lim }\text{  }\frac{M\left( \alpha \left( t\right) \right) }{1-\alpha \left( t\right) }\left( \frac{1}{\psi ^{\prime }\left( t\right) }\frac{d}{dt}\right) \text{ } _{1}I_{\gamma ,\beta }^{\alpha \left( t\right) ;\psi }f_{k}\left( t\right) 
\\
&=&\text{ }\frac{M\left( \alpha \left( t\right) \right) }{1-\alpha \left( t\right) }\left( \frac{1}{\psi ^{\prime }\left( t\right) }\frac{d}{dt} \right) \text{ }I_{\gamma ,\beta }^{\alpha \left( t\right) ;\psi }\underset{%
k\rightarrow \infty }{\lim }f_{k}\left( t\right)  \\ &=&^{RL}\mathfrak{D}_{a+}^{\alpha \left( t\right) ;\psi }\underset{k\rightarrow \infty }{\lim }f_{k}\left( t\right).
\end{eqnarray*}
\end{proof}

\begin{theorem} Let $0<\alpha \left( t\right) <1,$ $I=\left[ a,b\right] $ be a finite or infinite interval and $\psi \in C\left[ a,b\right] $ an increasing function such that $\psi ^{\prime }\left( t\right) \neq 0$, for all $t\in I.$ Assume that $\left( f_{k}^{\prime }\right) _{k=1}^{\infty }$ is a uniformly convergent sequence of continuous functions on $\left[ a,b\right] $ and $^{C}\mathfrak{D}_{a+}^{\alpha \left( t\right) ;\psi }f_{k}\left( \cdot \right) $ exist for every $k.$ Moreover assume that $\left( ^{C}\mathfrak{D}_{a+}^{\alpha \left( t\right) ;\psi }f_{k}\right) _{k=1}^{\infty }$ is convergent uniformly on $\left[ a+\varepsilon ,b\right] $ for every $\varepsilon >0.$ Then, for every $t\in \left[ a,b\right] $ we have
\begin{equation*}
\underset{k\rightarrow \infty }{\lim }\text{ }^{C}\mathfrak{D}_{a+}^{\alpha \left( t\right) ;\psi }f_{k}\left( t\right) =\text{ }^{C}\mathfrak{D}_{a+}^{\alpha \left( t\right) ;\psi }\underset{k\rightarrow \infty }{\lim }f_{k}\left( t\right) .
\end{equation*}
\end{theorem}

\begin{proof}The proof follows the same steps as the previous Theorem \ref{teo5}.
\end{proof}

The next Lemma \ref{le1}, we choose $\beta \rightarrow \gamma$ in $\mathbb{H}_{\gamma ,\beta }^{\alpha \left( t\right) ;\psi }\left( t_{0},\tau \right), $ then $\mathbb{H}_{\gamma }^{\alpha \left( t\right) ;\psi }\left( t_{0},\tau \right)$.

\begin{lemma}\label{le1} Let a function $f\in H^{1}\left( a,b\right) $ that attains its maximum at a point $t_{0}\in \left[ a,b\right] $ and $0<\alpha \left( t\right) <1.$ Then the inequality 
\begin{equation*}
^{C}\mathfrak{D}_{a+}^{\alpha \left( t_{0}\right) ;\psi }f\left( t_{0}\right) \geq \frac{M\left( \alpha \left( t_{0}\right) \right) }{1-\alpha \left( t_{0}\right) }\mathbb{H}_{\gamma }^{\alpha \left( t\right) ;\psi }\left( t_{0},\tau \right) \left[ f\left( t_{0}\right) -f\left( a\right) \right] \geq 0 
\end{equation*}
holds.
\end{lemma}

\begin{proof} We define the auxiliary function $g\left( t\right) =f\left( t_{0}\right) -f\left( t\right) ,$ $t\in \left[ a,b\right] .$ Then it follows that $g\left( t\right) \geq 0,$ on $\left[ a,b\right] ,$ $g\left( t_{0}\right) =g^{\prime }\left( t_{0}\right) =0$ and $^{C}\mathfrak{D}_{a+}^{\alpha \left( t_{0}\right) ;\psi }g\left( t\right) =-$ $^{C}\mathfrak{D}_{a+}^{\alpha \left( t_{0}\right) ;\psi }f\left( t\right) .$ Since $g\in H^{1}\left( a,b\right)$, $g^{\prime }$ is integrable and integrating by parts with $u=\mathbb{H}_{\gamma }^{\alpha \left( t_{0}\right) ;\psi }\left( t_{0},\tau \right) $ and $dv=g^{\prime }\left( \tau \right) d\tau $ yields 
\begin{eqnarray}\label{eq15}
&&^{C}\mathfrak{D}_{a+}^{\alpha \left( t_{0}\right) ;\psi }g\left( t_{0}\right) \notag\\ &=&\frac{M\left( \alpha \left( t_{0}\right) \right) }{1-\alpha \left( t_{0}\right) }\int_{a}^{t_{0}}\mathbb{H}_{\gamma }^{\alpha \left( t_{0}\right) ;\psi }\left(
t_{0},\tau \right) g^{\prime }\left( \tau \right) d\tau \notag  \\
&=&\frac{M\left( \alpha \left( t_{0}\right) \right) }{1-\alpha \left( t_{0}\right) }\left\{ -\mathbb{H}_{\gamma }^{\alpha \left( t_{0}\right) ;\psi }\left( t_{0},a\right) g\left( a\right) -\int_{a}^{t_{0}}\frac{d}{d\tau }\mathbb{H}_{\gamma }^{\alpha \left( t_{0}\right) ;\psi }\left( t_{0},\tau \right) g\left( \tau \right) d\tau \right\}. \notag\\ 
\end{eqnarray}
We recall that, for $0<\gamma <1,$ we have
\begin{equation*}
\mathbb{E}_{\gamma }\left( -t^{\gamma }\right) =\int_{0}^{\infty }e^{-rt}K_{\gamma }\left( r\right) dr,
\end{equation*}
where
\begin{equation*}
K_{\gamma }\left( r\right) =\frac{1}{\pi }\frac{r^{\gamma -1}sen\left( \gamma \pi \right) }{r^{2\gamma }+2r^{\gamma }\cos \left( \gamma \pi \right) +1}>0.
\end{equation*}

Note that, we can change the parameter $t$ inside the integral, since it acts as a constant, so we have
\begin{equation*}
\mathbb{E}_{\gamma }\left( -\psi \left( t\right) ^{\gamma }\right) =\int_{0}^{\infty }e^{-r\psi(t)}K_{\gamma }\left( r\right) dr.
\end{equation*}

Also, the parameter $\psi \left( t\right) $ doesn't influence the values of $K_{\gamma}(r)$. Thus,
\begin{eqnarray}\label{eq16}
&&\frac{d}{d\tau }\mathbb{H}_{\gamma }^{\alpha \left( t_{0}\right) ;\psi }\left( t_{0},\tau \right)\notag  \\ 
&=&\frac{d}{d\tau }\mathbb{E}_{\gamma }\left[ \left( \frac{\alpha \left( t_{0}\right) ^{1/\gamma }\left( \psi \left( t_{0}\right) -\psi \left(
\tau \right) \right) }{\left[ 1-\alpha \left( t\right) \right] ^{1/\gamma }} \right) ^{\gamma }\right]   \notag \\
&=&\int_{0}^{\infty }\frac{d}{d\tau }\exp \left[ -r\left( \frac{\alpha \left( t_{0}\right) }{1-\alpha \left( t\right) }\right) ^{1/\gamma }\left( \psi \left( t_{0}\right) -\psi \left( \tau \right) \right) \right] K_{\gamma
}\left( r\right) dr  \notag \\ 
&=&\left( \frac{\alpha \left( t_{0}\right) }{1-\alpha \left( t\right) } \right) ^{1/\gamma }\psi ^{\prime }\left( \tau \right) \int_{0}^{\infty }r\exp \left[ -r\left( \frac{\alpha \left( t_{0}\right) }{1-\alpha \left(
t\right) }\right) ^{1/\gamma }\left( \psi \left( t_{0}\right) -\psi \left( \tau \right) \right) \right] K_{\gamma }\left( r\right) dr>0.\notag\\
\end{eqnarray}

As $g\left( t\right) >0$ and Eq.(\ref{eq16}), we have that Eq.(\ref{eq15}) is nonnegative. We recall that $\mathbb{E}_{\gamma }\left( \psi \left( t\right) \right) >0,$ $0<\gamma <1,$ and thus
\begin{eqnarray*}
^{C}\mathfrak{D}_{a+}^{\alpha \left( t_{0}\right) ;\psi }g\left( t_{0}\right)  &\leq & \frac{M\left( \alpha \left( t_{0}\right) \right) }{1-\alpha \left( t_{0}\right) }\left[ -\mathbb{H}_{\gamma }^{\alpha \left( t_{0}\right) ;\psi }\left(
t_{0},a\right) g\left( a\right) \right]  \\
&=&-\frac{M\left( \alpha \left( t_{0}\right) \right) }{1-\alpha \left( t_{0}\right) }-\mathbb{H}_{\gamma }^{\alpha \left( t_{0}\right) ;\psi }\left( t_{0},a\right) \left[ f\left( t_{0}\right) -f\left( a\right) \right] \leq 0.
\end{eqnarray*}

Note that, $g\left( t\right) =f\left( t_{0}\right) -f\left( t\right)$ implies $\ -f\left( t_{0}\right) \leq g\left( t\right) \leq f\left( t_{0}\right) ,$ then
\begin{equation*}
^{C}\mathfrak{D}_{a+}^{\alpha \left( t_{0}\right) ;\psi }g\left( t_{0}\right) \geq \frac{M\left( \alpha \left( t_{0}\right) \right) }{1-\alpha \left( t_{0}\right) }\mathbb{H}_{\gamma }^{\alpha \left( t_{0}\right) ;\psi }\left( t_{0},a\right) \left[ f\left( t_{0}\right) -f\left( a\right) \right] \geq 0,
\end{equation*}
which proves the result.
\end{proof}

\begin{lemma}\label{lema2} Let a function $f\in H^{1}\left( a,b\right) $ then it holds that 
\begin{equation*}
^{C}\mathfrak{D}_{a+}^{\alpha \left( t\right) ;\psi }f\left( a\right) =0,\text{ }0<\alpha \left( t\right) <1.
\end{equation*}
\end{lemma}
\begin{proof} In fact, as $\mathbb{H}_{\gamma ,\beta }^{\alpha \left( t\right) ;\psi }\left( t,\tau \right) $ is continuous on $\left[ a,b\right] $, then it is on $L^{2}\left[ a,b\right] .$ Applying the Cauchy-Schwartz inequality we have
\begin{eqnarray}\label{eq17}
\left\vert ^{C}\mathfrak{D}_{a+}^{\alpha \left( t\right) ;\psi }f\left( t\right) \right\vert ^{2} &=&\left\vert \frac{M\left( \alpha \left( t\right) \right)  }{1-\alpha \left( t\right) }\int_{a}^{t}\mathbb{H}_{\gamma ,\beta }^{\alpha \left(
t\right) ;\psi }\left( t,\tau \right) f^{\prime }\left( \tau \right) d\tau \right\vert ^{2} \notag\\
&\leq &\frac{M^{2}\left( \alpha \left( t\right) \right) }{\left[ 1-\alpha \left( t\right) \right] ^{2}}\int_{a}^{t}\left[ \mathbb{H}_{\gamma ,\beta }^{\alpha \left( t\right) ;\psi }\left( t,\tau \right) \right] ^{2}d\tau \int_{a}^{t}\left[ f^{\prime }\left( \tau \right) \right] ^{2}d\tau .
\end{eqnarray}

Since $f\in H^{1}\left( a,b\right) $ then $f^{\prime }$ is square integrable and it holds $\int_{a}^{t}\left[ f^{\prime }\left( \tau \right) \right] ^{2}d\tau =0.$ On the other hand, the first integral in Eq.(\ref{eq17}) is bounded, which conclude the proof.
\end{proof}

\section{The nonlinear fractional differential equation} 
In this section, we display and discuss the comparison's principle, by means of a lemma. In this sense, we will present uniquenesses results for a nonlinear fractional differential equation and an estimate for its solution.

\begin{lemma}\label{lema3}\textnormal{(Comparison's Principle)} Let a function $u\in H^{1}\left( a,b\right) \cap C\left[ a,b\right] $, satisfying the fractional inequality 
\begin{equation}\label{eq18}
Q_{\alpha \left( t\right) }\left( u\right) =\text{ }^{C}\mathfrak{D}_{a+}^{\alpha \left( t\right) ;\psi }u\left( t\right) +q\left( t\right) u\left( t\right) \leq 0,\text{ }t>a,\text{ }0<\alpha \left( t\right) <1,
\end{equation}
where $q\left( t\right) \geq 0$ is continuous on $\left[ a,b\right] $ and $q\left( a\right) \neq 0.$ Then $u\left( t\right) \leq 0,$ $t\geq a$.
\end{lemma}

\begin{proof}In fact, for $u\in H^{1}\left( a,b\right) $ then by Lemma \ref{lema2} we have $^{C}\mathfrak{D}_{a+}^{\alpha \left( t\right) ;\psi }f\left( a\right) =0.$ By the continuity of the solution, the fractional inequality Eq.(\ref{eq18}), yields $ q(t)\leq 0,$ and hence $u(a)\leq 0$.

Assume, by contradiction, that the result is not true, because $u$ is continuous on $\left[ a,b\right] $ then $u$ attains absolute maximum at $t_{0}\geq a$ with $u\left( t_{0}\right) >0.$ Since $u(a)\leq 0$, then $\ t_{0}>a.$ Applying the result of Lemma \ref{le1}, we have 
\begin{equation*}
^{C}\mathfrak{D}_{a+}^{\alpha \left( t_{0}\right) ;\psi }u\left( t_{0}\right) \geq \frac{M\left( \alpha \left( t_{0}\right) \right) }{1-\alpha \left( t_{0}\right) }\mathbb{H}_{\gamma ,\beta }^{\alpha \left( t_{0}\right) ;\psi }\left( t_{0},\tau \right) \left[ u\left( t_{0}\right) -u\left( a\right) \right] >0.
\end{equation*}

We, then have
\begin{equation*}
^{C}\mathfrak{D}_{a+}^{\alpha \left( t_{0}\right) ;\psi }u\left( t_{0}\right) +q\left( t_{0}\right) u\left( t_{0}\right) \geq \text{ }^{C}\mathfrak{D}_{a+}^{\alpha \left( t_{0}\right) ;\psi }u\left( t_{0}\right) >0
\end{equation*}
which contradicts the fractional inequality Eq.(\ref{eq18}), which completes the proof. 
\end{proof}

The next two results, i.e., Lemma \ref{le4} refers to the unique solution of the non-linear differential equation, by means of the comparison's principle. As a consequence of Lemma \ref{le4}, we present Lemma \ref{le5}, which guarantees the limitation of the solution of the nonlinear differential equation.

Thus, the solution as follows involve the unique solution of a nonlinear fractional differential equation in which as an estimate for its solution is given.

\begin{lemma}\label{le4} Consider the nonlinear fractional differential equation
\begin{equation*}
^{C}\mathfrak{D}_{a+}^{\alpha \left( t\right) ;\psi }u\left( t\right) =f\left( t,u\right) ,\text{ }t>a,\text{ }0<\alpha \left( t\right) <1,
\end{equation*}
where $f\left( t,u\right) $ is a smooth function. If $f\left( t,u\right) $ is non-increasing with respect to $u$ then the above equation has a unique solution $u\in H^{1}\left( a,b\right)$.
\end{lemma}

\begin{proof} Assume that $u_{1},u_{2}\in H^{1}\left( a,b\right) $ be two solutions of the above equation and let $z=u_{1}-u_{2}$. Then it holds that 
\begin{equation*}
^{C}\mathfrak{D}_{a+}^{\alpha \left( t\right) ;\psi }z\left( t\right) =f\left(
t,u_{1}\right) -f\left( t,u_{2}\right).
\end{equation*}

Applying the mean value theorem we have
\begin{equation*}
f\left( t,u_{1}\right) -f\left( t,u_{2}\right) =\frac{\partial }{\partial u} f\left( u^{\ast }\right) z
\end{equation*}
which implies
\begin{equation}\label{eq19}
^{C}\mathfrak{D}_{a+}^{\alpha \left( t\right) ;\psi }z\left( t\right) -\frac{\partial }{\partial u}f\left( u^{\ast }\right) z\left( t\right) =0.
\end{equation}

Since $-\dfrac{\partial }{\partial u}f\left( u^{\ast }\right) >0,$ then $z\left( t\right) \leq 0,$ by virtue of Lemma \ref{lema3}. Also, Eq.(\ref{eq19}) holds for $-z$ and $-z\left( t\right) \leq 0,$ by virtue of Lemma \ref{lema3}. Thus, $z=0$ which proves that $u_{1}=u_{2}.$
\end{proof}

\begin{lemma}\label{le5} Consider the nonlinear fractional differential equation
\begin{equation}\label{eq20}
^{C}\mathfrak{D}_{a+}^{\alpha \left( t\right) ;\psi }u\left( t\right) =f\left( t,u\right) ,\text{ }t>a,\text{ }0<\alpha \left( t\right) <1
\end{equation}
where $f\left( t,u\right) $ is a smooth function. Assume that
\begin{equation*}
\lambda _{2}u+h_{2}\left( t\right) \leq f\left( t,u\right) \leq \lambda
_{1}u+h_{1}\left( t\right) ,
\end{equation*}
for all $t\in H^{1}\left( a,b\right) ,$ where $\lambda _{1},\lambda _{2}<0$. Let $v_{1}$ and $v_{2}$ be the solutions of
\begin{equation}\label{eq21}
^{C}\mathfrak{D}_{a+}^{\alpha \left( t\right) ;\psi }v_{1}\left( t\right) =\lambda_{1}v_{1}+h_{1}\left( t\right) ,\text{ }t>0,\text{ }0<\alpha \left( t\right)<1
\end{equation}
and
\begin{equation*}
^{C}\mathfrak{D}_{a+}^{\alpha \left( t\right) ;\psi }v_{2}\left( t\right) =\lambda _{2}v_{2}+h_{2}\left( t\right) ,\text{ }t>0,\text{ }0<\alpha \left( t\right) <1
\end{equation*}
respectively, then it holds that $v_{2}\left( t\right) \leq u\left( t\right) \leq v_{1}\left( t\right)$, $t\geq a$.
\end{lemma}

\begin{proof} We shall prove that $u\left( t\right) \leq v_{1}\left( t\right) $ and by applying analogous steps one can show that $v_{2}\left( t\right) \leq u\left( t\right)$. By subtracting Eq.(\ref{eq20}) from Eq.(\ref{eq21}) we have
\begin{equation*}
^{C}\mathfrak{D}_{a+}^{\alpha \left( t\right) ;\psi }u\left( t\right) -\text{ }^{C}\mathfrak{D}_{a+}^{\alpha \left( t\right) ;\psi }v_{1}\left( t\right) =f\left( t,u\right) -\lambda _{1}v_{1}-h_{1}\left( t\right) 
\end{equation*}
implies
\begin{eqnarray*}
^{C}\mathfrak{D}_{a+}^{\alpha \left( t\right) ;\psi }\left( u\left( t\right) -\text{ }v_{1}\left( t\right) \right)  &=&f\left( t,u\right) -\lambda _{1}v_{1}-h_{1}\left( t\right)   \notag \\ &\leq &\lambda _{1}\left( u-v_{1}\right).
\end{eqnarray*}
Let $z=u-v_{1},$ then it holds that,
\begin{equation*}
^{C}\mathfrak{D}_{a+}^{\alpha \left( t\right) ;\psi }z\left( t\right) -\text{ }\lambda _{1}z\left( t\right) \leq 0.
\end{equation*}
Since $\lambda _{1}<0,$ and $z\leq 0,$ by virtue of Lemma \ref{le4}, then $u\leq v_{1}$.
\end{proof}

\section{Concluding remarks}
We introduced two new non-singular kernel fractional derivative formulations, through the  $\psi-$Caputo and $\psi-$Riemann-Liouville fractional derivatives and the Riemann-Liouville fractional integral with respect to another function both with variable order. Here, it was possible to show that the new versions of fractional derivatives, besides admitting a class of fractional derivatives with non-singular kernel, from the choice of the function $\psi(\cdot)$ and the limit $\alpha(t)\rightarrow \alpha$ and $\alpha(t)\rightarrow 1$, it was also possible to present some important results involving uniformly convergent sequences of continuous functions and important results in the study of the limitation of solutions of fractional differential equations. On the other hand, we present a result of the comparison's principle, fundamental in the study of the nonlinear fractional differential equation presented in Section 4.

\bibliography{ref}
\bibliographystyle{plain}

\end{document}